\documentclass[pdflatex,sn-nature]{sn-jnl}

\usepackage{graphicx}%
\usepackage{multirow}%
\usepackage{amsmath,amssymb,amsfonts,mathtools}%
\usepackage{amsthm}%
\usepackage{mathrsfs}%
\usepackage[title]{appendix}%
\usepackage{xcolor}%
\usepackage{textcomp}%
\usepackage{manyfoot}%
\usepackage{booktabs}%
\usepackage{algorithmic}
\usepackage{algorithm}
\usepackage{listings}%
\usepackage{dsfont}
\usepackage{listings}%
\usepackage{cleveref}

\def\R{\mathds{R}}
\newcommand{\norm}[1]{\left\lVert#1\right\rVert}
\newcommand{\sclr}[1]{\left\langle#1\right\rangle}
\def\X{{\rm X}}

\newcommand{\sclrvX}[1]{\left\langle#1\right\rangle_{\vxLtwo}}
\def\E{\mathbb{E}}
\def\N{{\mathcal{N}_\delta}}
\def\NT{{\mathcal{N}_\delta^\bot}}
\newcommand{\prox}[1]{{\rm Prox}_{#1}}
\DeclareMathOperator*\argmin{\textrm{\textnormal{arg\,min}}}
\def\xd{{\textnormal{d}}}

\def\xCtwo{{\textnormal{C}}^{2}} 
 
\def\xCn#1{{\textnormal{C}}^#1}

\def\xLone{{\textnormal{L}}^{1}}
\def\xLtwo{{\textnormal{L}}^{2}} 
\def\xLinfty{{\textnormal{L}}^{\infty}} 
\def\xLn#1{{\rm L}^#1}

\def\vxLtwo{{\mathbb{L}}^{2}} 
\def\vxLinfty{{\mathbb{L}}^{\infty}} 
\def\vxLn#1{{\mathbb{L}}^#1}

\def\cons{{\rm \textbf{Cons}}}
\def\pv{{\rm \textbf{PV}}}

\def\buen{\bar{u}_{\epsilon_n}}
\def\byen{\bar{y}_{\epsilon_n}}
\def\bpen{\bar{p}_{\epsilon_n}}
\def\bmuen{\bar{\mu}_{\epsilon_n}}

\def\U{{\mathbb{U}}}
\def\vx{{\boldsymbol{x}}}
\def\vy{{\boldsymbol{y}}}
\def\vu{{\boldsymbol{u}}}
\def\vxi{{\boldsymbol{\xi}}}
\def\vz{{\boldsymbol{z}}}
\def\vlambda{{\boldsymbol{\lambda}}}
\def\vzeta{{\boldsymbol{\zeta}}}
\def\vQ{{\boldsymbol{Q}}}
\def\vPb{{\boldsymbol{P_b}}}
\def\vPm{{\boldsymbol{P_m}}}
\def\ind{{\rm i}}
\newcommand{\proj}[1]{{\rm Proj}_{#1}}

\theoremstyle{thmstyleone}%
\newtheorem{theorem}{Theorem}
\newtheorem{definition}{Definition}%
\newtheorem{problem}{Problem}
\newtheorem{assumption}{Assumption}%
\newtheorem{lemma}{Lemma}%
\newtheorem{remark}{Remark}

\begin{document}

\title[VPPHA]{Robust stochastic optimal control via variance penalization: Application to Energy Management Systems}


\author*[1]{\fnm{Paul} \sur{Malisani}}\email{paul.malisani@ifpen.fr}

\author[2]{\fnm{Adrien} \sur{Spagnol}}\email{adrien.spagnol@ifpen.fr}

\author[3]{\fnm{Vivien} \sur{Smis-Michel}}\email{vivien.smis-michel@ifpen.fr}


\affil*[1]{\orgdiv{Applied Mathematics Department}, \orgname{IFP Energies nouvelles}, \orgaddress{\street{1-4 Av. Bois Préau}, \city{Rueil-Malmaison}, \postcode{92852}, \state{Ile-de-France}, \country{France}}}

\affil[2]{\orgdiv{Applied Mathematics Department}, \orgname{IFP Energies nouvelles}, \orgaddress{\street{Rond-point de l'échangeur de Solaize}, \city{Solaize}, \postcode{69360}, \state{Rhône}, \country{France}}}

\affil[3]{\orgdiv{Control and Signal Processing Department}, \orgname{IFP Energies nouvelles}, \orgaddress{\street{Rond-point de l'échangeur de Solaize}, \city{Solaize}, \postcode{69360}, \state{Rhône}, \country{France}}}


\abstract{This paper addresses a class of robust stochastic optimal control problems. Its main contribution lies in the introduction of a general optimization model with variance penalization and an associated solution algorithm that improves out-of-sample robustness while preserving numerical complexity. The proposed variance-penalized model is inspired by a well-established machine learning practice that aims to limit overfitting and extends this idea to stochastic optimal control. Using the Douglas–Rachford splitting method, the authors develop a Variance-Penalized Progressive Hedging Algorithm (VPPHA) that retains the computational complexity of the standard PHA while achieving superior out-of-sample performance. In addition, the authors propose a three-step control framework comprising (i) a random scenario generation method, (ii) a scenario reduction algorithm, and (iii) a scenario-based optimal control computation using the VPPHA. Finally, the proposed method is validated through simulations of a stationary battery Energy Management System (EMS) using ground-truth electricity consumption and production measurements from a predominantly commercial building in Solaize, France. The results demonstrate that the proposed approach outperforms a classical Model Predictive Control (MPC) strategy, which itself performs better than the standard PHA.}




\maketitle

\section{Introduction}
This paper addresses the problem of robust stochastic optimal control for convex problems and its application to energy management. The robustness of stochastic optimal control algorithms is a key concern; indeed, as highlighted in \cite{smith_optimizers_2006,esfahani_data-driven_2017}, minimizing the expectation of an uncertain cost with respect to a probability measure estimated from empirical data can yield disappointing results when evaluated on out-of-sample data. In other words, the obtained results are not necessarily better than those produced by a standard MPC strategy. As noted by \cite{esfahani_data-driven_2017,smith_optimizers_2006}, \emph{this phenomenon is termed the optimizer's curse and is reminiscent of overfitting effects observed in statistics}. \\

This issue gave rise to the so-called distributionally robust stochastic optimization framework, which consists of solving problems of the form
\begin{equation}
    \label{eq:dist_rob}
    \inf_{\vx} \sup_{\nu \in\mathcal{P}(\Omega)}\int_\Omega f(\vx(\omega), \omega)\xd \nu(\omega),
\end{equation}
where $\mathcal{P}$ is a set of probability measures referred to as the ambiguity set. This set should be large enough to contain representative distributions but small enough to prevent the optimal solution from being overly conservative. For interested readers, \cite{rahimian_distributionally_2022} provides a comprehensive review of distributionally robust stochastic optimization.\\
In the context of multistage stochastic optimization, several works have focused on the robustness of optimization algorithms. In \cite{philpott_distributionally_2018}, the authors develop a distributionally robust Stochastic Dual Dynamic Programming (SDDP) algorithm where the ambiguity set is defined as follows
\[
\mathcal{P}_{\epsilon}(\mu) := \left\{ \sum_{s=1}^S \nu^s \delta_{\xi^s}: \sum_{s} (\mu^s - \nu^s)^2 \leq \epsilon, \nu^s \geq 0,\; \sum_{s=1}^S\nu^s=1\right\},
\]
with $\mu := \sum_{s=1}^S \mu^s \delta_{\xi^s}$ a scenario-based reference discrete probability. This framework was developed for a linear cost function and linear dynamics and is not easily extendable to nonlinear-cost-problems. In \cite{glanzer_incorporating_2019}, the authors define the ambiguity set using the so-called nested Wasserstein distance for stochastic processes \cite{pflug_0,pflug_pichler} and prove a large deviation result for the nested Wasserstein distance. However, as noted in \cite{de_oliveira_risk-averse_2021,shapiro_distributionally_2022}, using the nested distance to construct the ambiguity set is difficult when the stochastic processes are not stage-wise independent. One can instead use the standard Wasserstein distance to circumvent this difficulty.\\
In \cite{de_oliveira_risk-averse_2021}, the author proposes the Scenario Decomposition with Alternating Projections (SDAP) algorithm, an adaptation of the celebrated Douglas-Rachford algorithm \cite{DR1,DR2,DR3}, to tackle this distributionally robust optimization problem. Each iteration of the SDAP consists of solving a large Quadratic Programming (QP) optimization problem and a large number of independent optimization problems. Consequently, due to the QP solving, this method is numerically more demanding than the standard PHA introduced in \cite{rockafellar_scenarios_1991}. Moreover, the SDAP only considers robustness with respect to the weights of the probability measure, but not its support, which still makes it prone to overfitting. \\
In \cite{rockafellar_solving_2018}, the author proposes an adaptation of the standard PHA to address stochastic optimization problems involving risk measures. The proposed algorithm retains nearly the same numerical complexity as the standard PHA but requires the introduction of additional optimization parameters into the underlying deterministic problems. In the context of optimal control, these parameters are often handled by adding a trivial dynamic to the system to represent them, which significantly increases the size of the resulting optimal control problem \cite[Section~2.2.1]{bonnans_odes}.\\

In the context of linear regression for machine learning, the authors of \cite{blanchet_robust_2019,blanchet_quantifying_2017} showed that solving the distributionally robust optimization problem with an ambiguity set defined by the Wasserstein distance is equivalent to adding a variance-penalization term to the loss function. Inspired by this result, we aim to improve the robustness of scenario-based stochastic optimal control problems by penalizing their variance. The first contribution of this paper is the introduction of a variance-penalized optimization model and its associated solution algorithm, the Variance-Penalized Progressive Hedging Algorithm (VPPHA), which addresses the non-separability across scenarios. The second contribution consists of the development of a data-driven stochastic optimization framework that includes a scenario-generation algorithm inspired by \cite{amabile_optimizing_2021,thorey_2018}, a scenario-reduction method from \cite{heitsch_romisch}, and a VPPHA-based stochastic rolling-horizon strategy.\\
In \Cref{sec:notations}, we introduce the mathematical notations used throughout the article. In \Cref{sec:reg_pha}, we present the variance-penalized model, the principle of the VPPHA, and its proof of convergence in the context of convex optimization. In \Cref{sec:robust_soc}, we introduce a general stochastic constrained optimal control problem for linear systems and provide a general solving algorithm based on the VPPHA and the primal-dual deterministic optimal control algorithm from \cite{malisani_interior_2023,malisani_interior_2024}. In \Cref{sec:scenario_gen_and_red}, we present a general method to generate plausible electrical power consumption and photovoltaic production data from historical records based on \cite{amabile_optimizing_2021}, as well as the scenario-tree reduction algorithm used to compute a reduced set of representative scenarios developed in \cite{heitsch_romisch}. Finally, in \Cref{sec:numerical_results}, we combine the VPPHA control algorithm, the scenario generation, and the scenario-tree reduction methods, and compare the performance in terms of electricity bill reduction of the proposed method with those of a standard MPC and a standard PHA. This comparison is conducted by simulating the proposed EMS over 2 years using ground-truth electricity production and consumption data from a predominantly commercial building equipped with solar panels, thereby illustrating the relevance of our framework.

\section{Notations}
\label{sec:notations}

\subsection{On sets and spaces}
\begin{itemize}
    \item We denote by $\R,\R_+, \R_-, \bar{\R}$ the set of real numbers, non-negative numbers, non-positive numbers, and the extended real line, respectively. 
    \item The notation $0_{E}$ denotes the zero element of the vector space $E$. 
    \item We denote by $\xLn{p}([t_1,t_2];\R^m)$ the Lebesgue space of $p$-integrable functions from $[t_1,t_2]$ to $\R^m$, and we denote by $\norm{.}_{\xLn{p}}$ its corresponding norm. We denote $\xCn{p}([t_1, t_2] ; E)$ the set of $p$-times Fr\'echet differentiable functions from $[t_1, t_2]$ to $E$. 
\end{itemize}

\subsection{On functions}
\begin{itemize}
    \item Let $\X$ be a Hilbert space,  and let $f:{\rm X}\mapsto \R$ be Fr\'echet differentiable, we denote by $\nabla f : {\rm X} \mapsto {\rm X}$ the gradient of $f$. Let $\{{\rm X}_i, i=1,\dots,n\}$ be Hilbert spaces and let $f:{\rm X}_1\times \dots \times {\rm X}_n \mapsto \bar{\R}$ be Fr\'echet differentiable, we denote by $\nabla_{x_i}f: {\rm X}_1\times \dots \times {\rm X}_n \mapsto {\rm X}_i$ the Fr\'echet derivative of $f$ with respect to to the $i^{\rm th}$ variable.
    \item Let $g:\R^n\mapsto \R^m$ be differentiable, we denote by $g':\R^n \mapsto \R^{m \times n}$ the Jacobian of $g$. Let $g:\R^{n_1} \times \R^{n_2} \mapsto \R^m$, we denote by $g'_{x_1}:\R^{n_1} \times \R^{n_2} \mapsto \R^{m\times n_1}$ (resp. $g'_{x_2}:\R^{n_1} \times \R^{n_2} \mapsto \R^{m\times n_2}$) the Jacobian of $g$ with respect to the $n_1$ first  (resp. $n_2$ last) coordinates.
    \item For a function $h$ that only depends on time $t$, we denote $h_t$ its value at time $t$, by $h_{i,t}$ the value of its $i^{\rm th}$ component of $h$ is vector-valued, and by $\dot{h}$ its derivative. We denote by $u \mapsto y[u,y^0]$ the mapping associating to $u$ the solution of the differential equations $\dot{y}=f(y,u)$ with initial condition $y(0) = y^0$.
    \item Let ${\rm E}\subset {\rm X}$ a closed set, we denote by $\ind_{\rm E}:{\rm X}\mapsto \bar{\R}$, the indicator function of ${\rm E}$, i.e. $\ind_{\rm E}(x) = 0$ if $x \in {\rm E}$ and $\ind_{\rm E}(x) = +\infty$ otherwise. In addition, we denote by ${\rm Proj}_{\rm E}:{\rm X} \mapsto {\rm E}$ the orthogonal projection onto ${\rm E}$.
\end{itemize}

\subsection{On random variables}
\begin{itemize}
    \item Let $(\Omega, \mathcal{B}, \mu)$ be a probability space. We denote by $\E$ the corresponding mathematical expectation.
    \item Random variables are denoted in bold, such as $\vx$.
    \item We denote by $\vxLn{p}$ the space of random variables from $\Omega$ to $\xLn{p}$ respectively.
    \item We endow $\vxLn{p}$, with the norm $\norm{.}_{\vxLn{p}} := \E(\norm{.}_{\xLn{p}}^p)^{\frac{1}{p}}$. The space $\mathbb{L}^\infty$ is endowed with the following norm $\norm{\vxi}_{\vxLinfty} := \inf\{y\in\R: \mu(\{\omega\in\Omega:\norm{\vxi(\omega)}_{\xLinfty}>y\})=0\}$.
    \item The space $\vxLtwo$ is endowed with the inner product $\left \langle .,. \right \rangle_{\vxLtwo} := \E ( \left \langle .,. \right \rangle_{\xLtwo})$.
    \item Throughout the article, all random variables are defined on the same discrete probability space $(\Omega, \mathcal{B}, \mu)$, where $|\Omega| = S$, $\mathcal{B}=2^\Omega$, and given a random variable $\vxi$, we denote by $\xi^s$, $s=1,\dots,S$ its realizations also called scenarios.
\end{itemize}

\section{Robust stochastic optimization via variance penalization}
\label{sec:reg_pha}
\subsection{Problem presentation}
In this section, we present the general framework of multistage stochastic optimization problems. To do so, let us introduce the following definition.
 \begin{definition}[$\delta$-adaptation]
   Let $\vxi \in \vxLtwo([0,T]; \R^d)$, and $\vx\in \vxLtwo([0,T]; \R^m)$ be two random variables and denote $(\mathcal{F}_t)_{t\in[0,T]}$ the filtration generated for almost all times by the random variables $(\vxi(t))_{t\in[0,T]}$. Let $\delta \geq 0$, we denote
   \begin{equation}
       \vx \triangleleft_\delta \vxi \Longleftrightarrow \vx(t) = \E(\vx(t) \vert \mathcal{F}_{t-\delta}), \;{\rm a.e. }\; t\in[\delta,T],
   \end{equation}
   the property of $\vx$ being $\delta$-adapted to $\vxi$. We denote 
   \begin{equation}
       \label{eq:def_Nd}
       \N := \{\vx \in \vxLtwo([0,T];\R^m) : \vx \triangleleft_\delta \vxi\},
   \end{equation}
   the linear space of $\delta$-adapted variables and we denote $\proj{\N}: \vxLtwo([0,T];\R^m) \mapsto \N$ (resp. $\proj{\NT} : \vxLtwo([0,T];\R^m) \mapsto \NT$) the orthogonal projection on $\N$ (resp. $\NT$).
 \end{definition}
 The standard stochastic optimal control problem consists of solving the following problem.
\begin{problem}
    \label{prblm:general_msop}
    Let $f$ be a convex, proper, lower semi-continuous function. The stochastic optimal control problem we are interested in can be written as 
    \begin{equation}
    \label{eq:general_msop}
    \inf_{\vx\in \vxLtwo} \E\left[ f(\vx,\vxi)\right] +{\rm i}_\N(\vx)
\end{equation}
\end{problem}
This problem can be solved efficiently using the standard PHA from \cite{rockafellar_scenarios_1991}. However, the solutions of \Cref{prblm:general_msop} generally lack out-of-sample robustness, i.e., they are too fitted to the sampled scenarios $\xi^s$, $s=1,\dots,S$. To overcome this difficulty, we introduce a new optimization model in the next Section.

\subsection{Variance penalized model and solving algorithm}
To enhance the out-of-sample robustness of the stochastic optimization model, we adopt a variance-penalization approach from machine learning to compute a control strategy less prone to overfitting. The variance-penalized optimization model is as follows.
\begin{problem}[Variance Penalized Stochastic Optimization Model]
\label{prlbm:general_stoc_multistage_problem}
Let $\vxi \in \vxLtwo([0,T];\R^d)$ be a random variable, and let $\alpha \geq 0$. The variance-penalized stochastic optimal control problem we want to solve is now the following
\begin{equation}
    \label{eq:stoch_opt_reg}
    \inf_{\vx \in \vxLtwo} \E \left( f(\vx,\vxi)\right) +\frac{\alpha}{2} \norm{\vx - \E(\vx)}^2_{\vxLtwo} + {\rm i}_\N(\vx)
\end{equation}  
\end{problem}
Because of the variance term in the cost, the non-separability across scenarios is not limited to the indicator function ${\rm i}_{\N}$(.). Consequently, the Progressive Hedging Algorithm cannot be directly applied to this problem and must be adapted to account for the variance penalization term. This adaptation is the subject of \Cref{theorem:reg_pha}.
\begin{theorem}[Variance-Penalized PHA]
    \label{theorem:reg_pha}
    Let $\vlambda^0 \in \NT$, let $r>0$, and let $\alpha \geq 0$. If $f$ is convex, proper, and lower semi-continuous, the following sequence 
    \begin{subequations}
    \label{eq:reg_pha}
    \begin{align}
        \vx^{k+1} &
            \in \argmin_{\vx \in \vxLtwo} \E(f(\vx,\vxi)) + \sclrvX{\vlambda^k, \vx} 
            + \frac{r}{2} \norm{\vx - \proj{\N}(\vz^k)}^2_{\vxLtwo}, \label{eq:reg_pha_0} \\
        \vlambda^{k+1} & : = \vlambda^k + r \proj{\NT}(\vx^{k+1}),\label{eq:reg_pha_1}\\
        \vz^{k+1} & =
            \vz^k - \vx^{k+1}+ \frac{1}{r + \alpha} \left[\alpha \E(2 \vx^{k+1} - \vz^k)
            + r \proj{\N}\left(2 \vx^{k+1} - \vz^k\right) \right],
    \end{align}
\end{subequations}
    weakly converges to a fixed-point $(\bar{\vx}, \bar{\vlambda}, \bar{\vz})$ such that $\bar{\vx}$ is an optimal solution of \Cref{prlbm:general_stoc_multistage_problem}.
\end{theorem}
\begin{proof}
    First, let us split \cref{eq:stoch_opt_reg} as follows
    \begin{equation}
    \label{eq:def_splitting_drs}
        \begin{cases}
            \phi_\vxi(\vx) :=  \E \left( f(\vx,\vxi)\right)\\
            \psi(\vx) := \frac{\alpha}{2} \norm{\vx - \E(\vx)}^2_{\vxLtwo} + {\rm i}_\N(\vx)
        \end{cases}
    \end{equation}
    The Douglas-Rachford solving algorithm \cite{DR1,DR2,DR3} for this problem consists of finding a fixed-point of the following iterative procedure
    \begin{subequations}
        \label{eq:DR}
        \begin{align}
            \vx^{k+1} & = \prox{\frac{1}{r}\phi_\vxi}(\vz^k) \label{eq:DR1}\\
            \vz^{k+1} & = \vz^k + \prox{\frac{1}{r}\psi}(2 \vx^{k+1} - \vz^k) - \vx^{k+1} \label{eq:DR2}
        \end{align}
    \end{subequations}
    The proof of \Cref{theorem:reg_pha} consists of proving that the scheme from  \cref{eq:reg_pha} is the Douglas-Rachford scheme from \cref{eq:DR} applied to the optimization problem from \cref{eq:stoch_opt_reg} using the splitting defined in \cref{eq:def_splitting_drs}.\\
    Now, let us compute $\prox{\frac{1}{r}\psi}$
\begin{equation*}
    \prox{\frac{1}{r}\psi}(\vz) := \argmin_{\vx\in \vxLtwo} \frac{\alpha}{2}\norm{\vx - \E(\vx)}^2_{ \vxLtwo}  + {\rm i}_{\N}(\vx)
    + \frac{r}{2} \norm{\vx - \vz}^2_{ \vxLtwo}
\end{equation*}
    Let $L:\vxLtwo \times \NT$ be the Lagrangian of this problem, defined as follows
    \[
    L(\vx, \vlambda) := \frac{r}{2} \norm{\vx - \vz}^2_{\vxLtwo} + \frac{\alpha}{2}  \norm{\vx - \E(\vx)}^2_{\vxLtwo} + \left \langle \vlambda, \proj{\NT}(\vx) \right\rangle_{\vxLtwo}
    \]
    The Lagrangien is strictly convex in $\vx$ and since $\N$ is non-empty closed linear subset of $\vxLtwo([0,T];\R^m)$, strong duality holds. Recalling that $\proj{\NT}$ is self-adjoint, a direct calculus yields that the Lagrangian's saddle-point $(\bar{\vx}, \bar{\vlambda})$ satisfies
    \begin{subequations}
        \begin{align}
            \nabla_{\vx}L(\bar{\vx}, \bar{\vlambda}) &= r(\bar{\vx} - \vz) +\alpha(\bar{\vx} - \E(\bar{\vx})) + \bar{\vlambda} = 0\label{eq:lag_vx_prox_g_part_1}\\
            \nabla_{\vlambda}L(\bar{\vx}, \bar{\vlambda}) &= \proj{\NT}(\bar{\vx}) = 0\label{eq:lag_vlambda_prox_g_part_2}
        \end{align}
    \end{subequations}
    Since $\bar{\vlambda}\in \NT$, we have $\E(\bar{\vlambda})=0$, thus taking the expectation in \cref{eq:lag_vx_prox_g_part_1} yields
    \begin{equation}
        \label{eq:lag_vlambda_prox_g_part_3}
        \E(\bar{\vx}) = \E(\vz)
    \end{equation}
    Incorporating \cref{eq:lag_vlambda_prox_g_part_2,eq:lag_vlambda_prox_g_part_3} into \cref{eq:lag_vx_prox_g_part_1} and decomposing the resulting relation into its components in $\N$ and $\NT$ gives
        \begin{align*}
            r(\bar{\vx} - \proj{\N}(\vz)) + \alpha\left(\bar{\vx} - \E(\vz)\right) &= 0\\
             \bar{\vlambda} & = r\proj{\NT}(\vz)
        \end{align*}
    Which, in turn, proves that
    \begin{equation}
        \label{eq:prox_g_reg}
         \prox{\frac{1}{r}\psi}(\vz) = \bar{\vx} = \frac{\alpha \E(\vz) + r{\rm Proj}_{\N}(\vz)}{r + \alpha}
    \end{equation}
Therefore, $\prox{\frac{1}{r}\psi}(.) \in \N$. Now, define $\vlambda^k := - r \proj{\NT}(\vz^k)$, then, using \cref{eq:DR2}, we have
\begin{align}
        \vlambda^{k+1}&= -r \proj{\NT}(\vz^k - \vx^{k+1} + \prox{\frac{1}{r}\psi}(2 \vx^{k+1} - \vz^k)) \nonumber\\
        &= -r \proj{\NT}(\vz^k) + r\proj{\NT}(\vx^{k+1}) \nonumber\\
        &= \lambda^k + r\proj{\NT}(\vx^{k+1})\label{eq:dr_recurvive_lbd_reg}
\end{align}
Now, let us compute $\prox{\frac{1}{r}\phi(.,\vxi)}$
\begin{align}
    \prox{\frac{1}{r}\phi_\vxi}(\vz^k) &=
        \argmin_{\vx\in \vxLtwo} \E(f(\vx,\vxi))
        + \frac{r}{2}\norm{\vx - \vz^k}^2_{\vxLtwo}\nonumber \\
                            & = 
        \argmin_{\vx\in \vxLtwo} \E(f(\vx,\vxi)) 
        + \frac{r}{2}\norm{\vx - \proj{\N}(\vz^k) - \proj{\NT}(\vz^k)}^2_{\vxLtwo} \nonumber\\
                            & = \begin{multlined}[t]\argmin_{\vx\in \vxLtwo} \E(f(\vx,\vxi)) - r \sclrvX{\vx, \proj{\NT}(\vz^k)}
                            + \frac{r}{2}\norm{\vx - \proj{\N}(\vz^k)}^2_{\vxLtwo}\\
                            + \frac{r}{2}\norm{\proj{\NT}(\vz^k)}^2_{\vxLtwo} \end{multlined}\nonumber\\
                            & = \argmin_{\vx\in \vxLtwo} \E(f(\vx,\vxi)) + \sclrvX{\vx,\vlambda^k}
                            + \frac{r}{2}\norm{\vx - \proj{\N}(\vz^k)}^2_{\vxLtwo} \label{eq:last_proxf_z}
\end{align}
The transition to the last line stems from noting that $\norm{\proj{\NT}(\vz^k)}^2_{\vxLtwo}$ does not depend on $\vx$, thus has no influence on the $\argmin$ and can be ignored. Finally, using \cref{eq:prox_g_reg,eq:dr_recurvive_lbd_reg,eq:last_proxf_z}, it is straightforward to check that the scheme from \cref{eq:reg_pha} is exactly the Douglas-Rachford algorithm from \cref{eq:DR} applied to \Cref{prlbm:general_stoc_multistage_problem}, which concludes the proof.
\end{proof}

\begin{remark}
One can check that the algorithm from \Cref{theorem:reg_pha} with $\alpha = 0$ is equivalent to the standard PHA from \cite{rockafellar_scenarios_1991}.
\end{remark}

\section{Robust Stochastic Optimal Control}
\label{sec:robust_soc}
\subsection{Problem presentation}
\begin{problem}[Stochastic optimal control problem]
\label{prblm:general_socp}
The problem we are interested in consists of solving the following stochastic optimal control problem
\begin{equation}
    \label{eq:cost_general_prb}
    \min_{\vu \in \U} \E\left[\int_0^T\ell(\vy_t, \vu_t, \vxi_t) \xd t+ h(\vy_T) \right]
\end{equation}
Where $\U \subseteq \vxLtwo([0,T];\R^m)$ is the space of random variables such that, for all $\vu \in \U$, the following holds
\begin{subequations}
        \begin{align}
            \dot{\vy}_t & =
        A_t \vy_t + B_t \vu_t \textrm{ a.s.} \label{eq:dyn_const}\\
        0_{\R^p} &\geq C_t \vy_t + D_t \vu_t + E_t \textrm{ a.s.}\label{eq:ineq_const}\\
        \vy_0 & = y^0 \textrm{ a.s.} \label{eq:y0_const}\\
         0_{\R^q} & = F \vy_T + G \textrm{ a.s.}  \label{eq:yT_const}\\
         \vu & \in \N \label{eq:meas_const}
        \end{align}
    \end{subequations}

\end{problem}
 In this general setting, \cref{eq:meas_const} encompasses both the Decision-Hazard and Hazard-Decision frameworks, even though this paper's application falls under the Decision-Hazard framework since $\delta>0$. Finally, the problem is solved under the following assumptions.
\begin{assumption}
 \label{ass:prob_optim}
 The data of the problem satisfy the following assumptions
 \begin{enumerate}
     \item[$i)$] The function $\ell \in \xCtwo(\R^n\times\R^m \times \R^d;\R)$ is proper, and convex with respect to the first two variables. The function $h\in \xCtwo(\R^n;\R)$ is also proper and convex.
     \item[$ii)$] There exists $R<+\infty$ such that for all $(\vy, \vu)$ satisfying \cref{eq:dyn_const,eq:ineq_const,eq:y0_const,eq:yT_const}, we have 
\begin{equation}
\norm{\vu}_{\vxLinfty}\leq R
\end{equation}
\item[$iii)$] The mappings $A, B, C, D, E$ are in $\xLinfty$. 
 \end{enumerate} 
\end{assumption}

\subsection{VPPHA implementation for \Cref{prblm:general_socp}}
In this section, we present a detailed account of VPPHA's implementation for solving \Cref{prblm:general_socp}. To do so, let us introduce the following definition
\begin{definition}
    \label{definition:BR}
    Let $\vxi \in \vxLtwo([0,T];\R^d)$, $\vlambda, \vzeta \in \vxLtwo([0,T];\R^m)$ be three random variables. We denote ${\rm SOCP}(\vxi,\vzeta,\vlambda) \in \vxLtwo([0,T];\R^m)$ the random variable defined scenario by scenario as follows 
    \begin{equation}
        \label{eq:def_br}
        {\rm SOCP}(\vxi, \vzeta, \vlambda)^s := \bar{u}^s \;\;\forall s =1,\dots,S,
    \end{equation}
    where $\bar{u}^s$ is the solution of the following deterministic optimal control problem
    \begin{equation}
         \label{eq:cost_pha_scen}
         \min_{u \in \xLtwo([0,T];\R^m)} \int_0^T \ell(y_t, u_t,\xi^s_t) \xd t
         + \sclr{\lambda^s, u}_{\xLtwo} + \frac{r}{2}\norm{u - \zeta^s}_{\xLtwo}^2 + h(y_T),
     \end{equation}
 under constraints from \cref{eq:dyn_const,eq:ineq_const,eq:y0_const,eq:yT_const}.
\end{definition}

Using this definition, we can now write the Variance-Penalized PHA for \Cref{prblm:general_socp}.

\begin{theorem}
     Let $\vxi \in \vxLtwo([0,T];\R^d)$ be a discrete random variable, let $\vlambda^0 \in \NT$, let $r>0$, let $\alpha \geq 0$, and assume that \Cref{ass:prob_optim} holds, then the following sequence
    \begin{subequations}
        \label{eq:reg_pha_ipm}
        \begin{align}
            \vu^{k+1}&:= {\rm SOCP}(\vxi, \proj{\N}(\vz^k), \vlambda^k),\label{eq:reg_pha_ipm_1}\\
            \vlambda^{k+1} & := \vlambda^k + r \proj{\NT}(\vu^{k+1}),\label{eq:reg_pha_ipm_2}\\
        \vz^{k+1} & = 
            \vz^k - \vu^{k+1}+\frac{1}{r + \alpha} \left[\alpha \E(2 \vu^{k+1} - \vz^k)
            + r \proj\N\left(2 \vu^{k+1} - \vz^k\right)\right],\label{eq:reg_pha_ipm_3}
        \end{align}
    \end{subequations}
    is the Variance-Penalized PHA described in \Cref{theorem:reg_pha} applied to \Cref{prblm:general_socp}.
\end{theorem}\begin{proof}
To do so, we only need proving that there exists a function denoted $\varphi:\xLtwo([0,T];\R^m)\times \xLtwo([0,T];\R^d) \mapsto \bar{\R}$, proper, convex, and lower semicontinuous embedding the cost and constraints of \Cref{prblm:general_socp}, and such that \cref{eq:reg_pha_0} applied to this function gives \cref{eq:reg_pha_ipm_1}. To do so, let us define the function
$\varphi:\xLtwo([0,T];\R^m)\times \xLtwo([0,T];\R^d) \to \bar{\R}$ as follows
\begin{align}
    \varphi(u,\xi^s) := &\int_{0}^T \ell(y_t[u,y^0], u_t, \xi^s_t) \, \xd t  + h(y_T[u,y^0])\nonumber \\
    &+ \int_0^T {\rm i}_{\{0\}}\!\left(\dot{y}_t[u,y^0] - A_t y_t[u,y^0] - B_t u_t\right) \xd t\nonumber \\
    &+ \int_0^T {\rm i}_{\R_-^p}\!\left(C_t y_t[u,y^0] + D_t u_t + E_t\right) \xd t\nonumber \\
    &+ {\rm i}_{\{0\}}\!\left(F y_T[u,y^0] + G\right),\label{eq:def_f_for_equiv_proof}
\end{align}
where $y[u,y^0]$ denotes the mapping that associates a control $u$ with the corresponding state trajectory solving \cref{eq:dyn_const,eq:yT_const}. Since the dynamics in \cref{eq:dyn_const} are linear, this mapping is linear as well. Moreover, by \Cref{ass:prob_optim}, the function $\ell$ is convex with respect to its first two arguments. Consequently, $\varphi$ is proper, convex, and lower semicontinuous, as it is an integral of compositions of proper, convex, and lower semicontinuous functions with affine mappings. Using this definition of $\varphi$, it is straightforward to check that solving \Cref{prblm:general_socp} using the Variance-Penalized PHA consists of solving
$$
\min_{\vu \in \xLtwo([0,T];\R^m)}\E\left(\varphi(\vu, \vxi) \right) + \frac{\alpha}{2}\norm{\vu - \E(\vu)}^2_{\vxLtwo} + \ind_{\N}(\vu)
$$
Now, let us denote $\vzeta^k:=\proj{\N}(\vz^k)$, then \cref{eq:reg_pha_0}, for our problem becomes
\[
\vu^{k+1} \in \argmin_{\vu\in \vxLtwo} 
\sum_{s=1}^S \mu_s \left[ 
\varphi(u^s,\xi^s) + \sclr{(\lambda^k)^s, u^s}_{\xLtwo} + \frac{r}{2} \norm{u^s - (\zeta^k)^s}_{\xLtwo}^2
\right].
\]
This problem is separable across scenarios. For each scenario $s = 1,\dots,S$, we therefore solve
\[
\min_{u\in\xLtwo} 
\varphi(u, \xi^s) + \sclr{(\lambda^k)^s, u}_{\xLtwo} + \frac{r}{2} \norm{u - (\zeta^k)^s}_{\xLtwo}^2.
\]
Thus, using \Cref{definition:BR,eq:def_f_for_equiv_proof}, we have
$$
\vu^{k+1} = {\rm SOCP}(\vxi, \proj{\N}(\vz^k), \vlambda) ,
$$
which concludes the proof.
\end{proof}
 \subsection{Computing ${\rm SOCP}(\vxi, \vzeta, \vlambda)$}
 \label{sec:deterministic_ocp}
 Now, computing ${\rm SOCP}(\vxi, \vzeta, \vlambda)$ consists of solving $S$ deterministic optimal control problems as described in \Cref{definition:BR}. To do so, one needs an efficient optimal control problem solver able to handle pure-state, mixed, and control constraints. To do so, we use the primal-dual implementation of interior-point methods for optimal control problems described in \cite{malisani_interior_2023,malisani_interior_2024}. This primal-dual algorithm is well-suited to stochastic optimal control problems due to its numerical efficiency and ability to handle pure-state constraints, which are notoriously difficult. We have the following convergence result.
 \begin{lemma}
    \label{lmm:primal_dual}
     Let $(\epsilon_n)_n$ be a decreasing sequence of positive parameters converging to zero, and let $(\buen^s, \byen^s, \bpen^s, \bmuen^s,\bar{\eta}_{\epsilon_n}^s)$ be a solution of the following two-point boundary value problem
     \begin{subequations}
        \label{eq:tpbvp}
         \begin{align}
         \dot{y}_t & = A_t y_t + B_t u_t,\\
         \dot{p}_t & = -\ell'_y(y_t,u_t,\xi^s_t) - A_t^\top p_t - C_t^\top \mu_t,\\
         0_{\R^m} &= \ell'_u(y_t,u_t,\xi^s_t) + \lambda^s_t +r(u_t - \zeta^s_t)
         + B_t^\top p_t + D_t^\top \mu_t,\\
         0_{\R^p} & = {\rm FB}(\mu_t, C_t y_t + D_t u_t + E_t,\epsilon_n),\\
         0_{\R^n} & = y_0 - y^0,\\
         0_{\R^q} & = F y_T + G ,\\
         0_{\R^n} & = p_T - \nabla h(y_T) - F^\top \eta ,
     \end{align}
     \end{subequations}
    where ${\rm FB}(x,y,\epsilon) := x - y - \sqrt{x^2 + y^2 + 2 \epsilon}$. Then the sequence $(\buen^s)_n $ converges to $\bar{u}^s:={\rm SOCP}(\vxi, \vzeta, \vlambda)^s$.
 \end{lemma}
 \begin{proof}
    First, the set
    $$
    \left\{u \in \xLtwo([0,T];\R^m): \textnormal{\cref{eq:dyn_const,eq:ineq_const,eq:y0_const,eq:yT_const}  hold} \right\},
    $$
    is convex since the dynamics from \cref{eq:dyn_const} are linear and the constraints from \cref{eq:ineq_const,eq:y0_const,eq:yT_const} are affine. Let $y[u,y^0]$ be the solution of \cref{eq:dyn_const,eq:y0_const}. From \Cref{ass:prob_optim}, the function $\ell$ is convex with respect to $(y,u)$, therefore, the mapping $u_t \mapsto \ell(y_t[u,y^0], u_t,\xi^s_t)$ is linear. Integration with respect to the time variable preserves the convexity, which proves that the mapping
     \begin{equation}
         \xLtwo([0,T];\R^m) \ni u \mapsto \int_0^T\ell(y_t[u,y^0], u_t,\xi^s_t)\xd t
         + \sclr{\lambda^s, u}_{\xLtwo} + \frac{r}{2}\norm{u - \zeta^s}_{\xLtwo}^2 \in \R,
     \end{equation}
     is strictly convex for all $r>0$. Thus, the deterministic optimal control problem for computing ${\rm SOCP}$ described in \cref{eq:cost_pha_scen} is strictly convex and thus has a unique optimal solution. In addition, from \cite[Corollary 6.1.]{malisani_interior_2024}, the sequence $(\buen^s, \byen^s, \bpen^s, \bmuen^s,\bar{\eta}_{\epsilon_n}^s)_n$ converges to a point $(\bar{u}^s, \bar{y}^s, \bar{p}^s, \bar{\mu}^s,\bar{\eta}^s)$ satisfying the first-order conditions of optimality. Using the uniqueness of the optimal solution of \cref{eq:cost_pha_scen}, necessarily $\bar{u}^s$ is the unique optimal solution. Now, from \cite[Corollary 6.1.]{malisani_interior_2024}, the convergence of $\buen^s$ is in the $\xLone$-topology. Now, using \Cref{ass:prob_optim}, we have
     \begin{multline}
         \lim_{n\rightarrow \infty}\norm{\buen^s - \bar{u}^s}_{\xLtwo}^2 \leq \lim_{n\rightarrow \infty}\norm{\buen^s - \bar{u}^s}_{\xLinfty} \norm{\buen^s - \bar{u}^s}_{\xLone}
         \\
         \leq \lim_{n\rightarrow \infty}2R\norm{\buen^s - \bar{u}^s}_{\xLone}=0, 
     \end{multline}
     which proves that the convergence of the optimal control, as required, also holds in the $\xLtwo$-topology, which concludes the proof.
 \end{proof}

\section{Reduced scenario tree generation}
\label{sec:scenario_gen_and_red}
\subsection{Scenario generation}
To conduct stochastic optimization, we must provide a sufficient number of scenarios to account for possible day-to-day variability. Using historical data from a building, we follow the method proposed by \cite{amabile_optimizing_2021} to generate plausible scenarios with respect to the underlying distribution of the measurements. The building mentioned above is a predominantly commercial three-story building located in the city of Solaize, France. The top two floors are offices, and the ground floor houses a small glass factory that operates occasionally.\\
First and foremost, if necessary, the available data are clustered into groups based on a priori criteria, such as seasonal or day specificity. Then, for each dataset group, the measurements are normalized to a maximum of 1 using a scaling factor equal to the peak value observed within the cluster. By definition, the minimum value is 0, since electrical production or consumption is always non-negative. \\
Then, for each dataset group and a given number of timesteps per hour (1, 2, or 6), we directly compute the quantiles from the ground truth measurements rather than relying on quantile regression forecasts, as in \cite{thorey_2018}. Thus, for a quantile level $\alpha \in [0,1]$ and a list of measurements at the timestamp $t\in[0, 24)$, $x_1^t, \ldots, x_n^t \in \mathbb{R}$, the $\alpha$ quantile is

\begin{equation*}
    Q_{\alpha}^t (x_1^t, \ldots, x_n^t) = x^t_{(\lceil n\alpha \rceil)}, 
\end{equation*}
with $x^t_{(i)}$ the $i$th order statistic of the list $(x_1^t, \ldots, x_n^t)$. In other words, the $\alpha$ quantile is the $\lceil n\alpha \rceil$-th smallest value of $x_1^t, \ldots, x_n^t$. Obtaining the $\alpha$ quantile for every possible timestamp yields quantile curves such as those in \Cref{fig:quantile_curves} for PV production and in \Cref{fig:quantile_conso} for the building's electrical consumption. 
\begin{figure}
    \centering
    \includegraphics[width=\columnwidth]{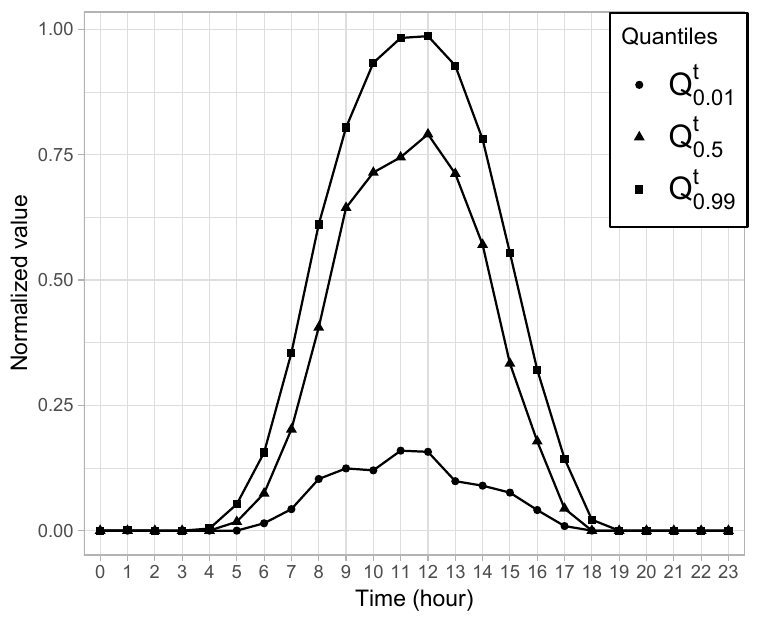}
    \caption{Quantiles curves obtained for $\alpha \in \{0.01, 0.5, 0.99\}$ for electrical production.}
    \label{fig:quantile_curves}
\end{figure}
\begin{figure}
    \centering
    \includegraphics[width=\columnwidth]{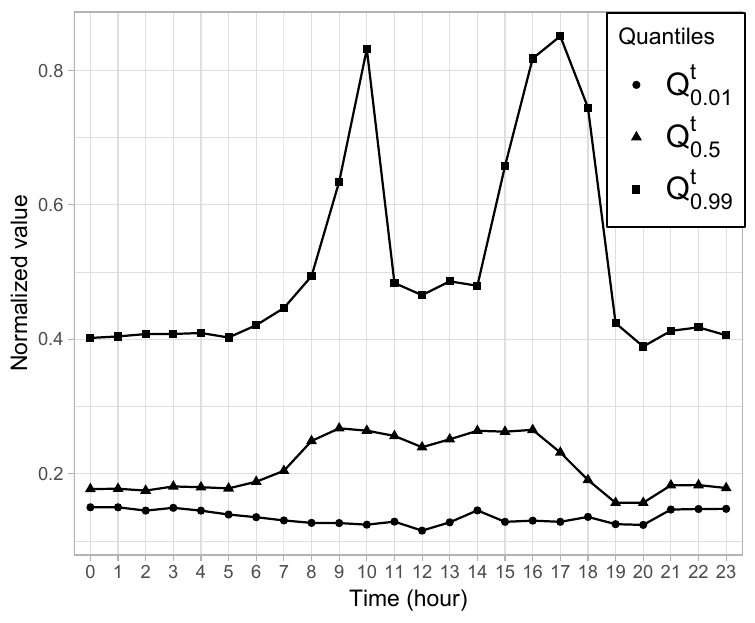}
    \caption{Quantiles curves obtained for $\alpha \in \{0.01, 0.5, 0.99\}$ for electrical consumption. The pics in consumption are due to the occasional operations of the building's glass factory.}
    \label{fig:quantile_conso}
\end{figure}
The lower and upper curves are the 0.01 and 0.99 quantile profiles, respectively. It means that only 1\% of the data is below the primer, and 99\% is above it at any timestep. We build 19 additional quantile profiles between 0.05 and 0.95, with a 0.05 increment, for a total of 21 curves. We can build an empirical cumulative distribution function using the different order quantiles.\\
Then, to generate a single scenario, we follow \cite{amabile_optimizing_2021} and, instead of drawing individual values from their respective cumulative distribution functions, we introduce correlation between consecutive timesteps. Assuming two random variables $\vx_k$ and $\vx_{k+1}$ of respective cumulative distribution $F_{k}$ and $F_{k+1}$, the following stochastic process is used to generate the scenarios: 

\begin{equation*}
    \begin{cases}
        \vx_{k+1} = F^{-1}_{k+1}(F(1-\alpha)F_k(\vx_k)+\alpha \vu_{k+1})\text{, for } k > 0 \\
        x_0 = 0
    \end{cases}
\end{equation*}
with $\alpha \in ]0,1[$, $\vu_{k+1} \sim \mathcal{U}(0,1)$ and $F$ a cumulative distribution function defined by
\begin{equation}
    F(x) = 
    \begin{cases}
       \frac{x^2}{2ab} & \text{if } 0 \leq x \leq a \\
        \frac{a}{2b} + \frac{x-a}{b} & \text{if } a \leq x \leq b \\
        \frac{a}{2b} + \frac{b-a}{b} + \frac{x - b - \frac{x^2 - b^2}{2}}{ab}& \text{if } b \leq x \leq 1
    \end{cases}
    \label{eq:cdf_F}
\end{equation}
with $a = \min (\alpha, 1 - \alpha)$ and $b = \max (\alpha, 1 - \alpha)$. \Cref{eq:cdf_F} is the cumulative distribution of a random variable defined as the following  weighted sum 
\begin{equation}
    {\bf W} = (1-\alpha)\vu_k + \alpha \vu_{k+1}
    \label{eq:W}
\end{equation}
with $\vu_k = F_k(\vx_k)$ and $\vu_{k+1}=F_{k+1}(\vx_{k+1}) \sim \mathcal{U}(0,1)$ by definition of the Probability Integral Transform. They use the property that $F_{k+1}^{-1}(F({\bf W}))$ has the same probability density function as $\vx_{k+1}$ but also encompasses a degree of correlation with $\vx_k$ by definition of \cref{eq:W}. This degree of correlation is directly affected by $\alpha$. \\
In our study, the parameter $\alpha$ is optimized within each cluster via a grid search to minimize the average prediction error across a portfolio of known scenarios, generating a reasonable number of trajectories.

\subsection{Scenario reduction}
To solve \Cref{prblm:general_socp} using the algorithm from \Cref{theorem:reg_pha}, one must make a trade-off between the number of scenarios and the numerical tractability of the problem, i.e., between the quality of the uncertainties representation and the numerical tractability. One way to achieve such a trade-off is to generate a large number of equiprobable scenarios, denoted $N_s$, and to derive from these $N_{\rm red} < N_s$ scenarios and their associated probabilities, such that the reduced set minimizes the Wasserstein distance to the original set of scenarios. We perform this task using the so-called fast-forward selection method from \cite[Algorithm 2.4]{heitsch_romisch}.

\section{Numerical example}
\label{sec:numerical_results}
\subsection{Stochastic optimal control of a stationary battery}
The problem we are interested in is the optimal control of a stationary battery connected downstream of a prosumer's meter, i.e., a customer with uncontrollable electrical production and consumption sources. The schematic diagram of such an installation is displayed in \Cref{fig:scheme_principle}. The stochastic optimal control problem consists of minimizing the following cost
\begin{equation}
    \label{eq:cost_battery_init}
    \inf_{\vQ,\vPb\in \vxLinfty\times \vxLtwo} \E\bigg[\int_0^T {\rm pr}_{\rm b}(t) \max\{\vPm(t),0\} 
    +  {\rm pr}_{\rm s}(t) \min\{\vPm(t),0\} \xd t\bigg]
\end{equation}
where ${\rm pr}_{\rm b}$ (resp. ${\rm pr}_{\rm s}$) is the buying (resp. selling) price of electricity satisfying $0 \leq {\rm pr}_{\rm s} (t)\leq {\rm pr}_{\rm b}(t)$ at all times, and $\vPm$ is the power measure at the meter. This power is defined as follows
\begin{equation}
    \label{eq:def_pm}
    \vPm := \cons - \pv +\frac{1}{\rho_c} \max\{\vPb,0\} + \rho_d \min\{\vPb,0\} 
\end{equation}
where $\cons$ (resp. $\pv$) is the uncontrollable electric consumption (resp. production), $\rho_c, \rho_d = 0.9$ are respectively the battery charge and discharge efficiencies. The battery's dynamics are as follows
\begin{equation}
    \label{eq:battery_dyn}
    \dot{\vQ}(t) = \vPb(t)
\end{equation}
The stochastic optimal control problem is solved under the following constraints
\begin{align}
    \vQ & \in \vxLinfty([0,T];[0,13])\\
    \vPb & \in \vxLtwo([0,T];[-8,8\rho_c])\\
    \vQ(0), \vQ(T) & = Q^0\label{eq:boundary_constraints}
\end{align}
At this point, due to the $\max$ and $\min$ functions in \cref{eq:cost_battery_init,eq:def_pm}, requirements from \Cref{ass:prob_optim} are not satisfied. To overcome this difficulty, these functions are replaced by their smooth approximations defined as follows
\begin{align*}
    {\rm max}_\mu(x,y) &:= \frac{1}{2} \left(x + y + \sqrt{(x - y)^2 + \mu} \right)\\
    {\rm min}_\mu(x, y) &:= \frac{1}{2} \left(x + y - \sqrt{(x - y)^2 + \mu} \right)
\end{align*}
and we set $\mu=10^{-5}$ to conduct all the computations. Finally, let us discuss the non-anticipativity constraint. The random processes $\cons$, $\pv$ are time-discrete periodic measures at the meter. Let $(t_0, t_1,\dots, t_N)$ be the time sequence of measures at meter satisfying $t_0:=0$, $t_N := T$ and, for all $k$, $\delta := t_{k+1} - t_k = $ 10 minutes. At time $t_k$, the value $\cons(t_k)$ (resp. $\pv(t_k)$) corresponds to the mean consumption (resp. production) power on the interval $[t_k, t_{k+1})$. Hence, $\cons(t_k)$ (resp. $\pv(t_k)$) is known at $t_{k+1} = t_k + \delta$. Therefore, the problem at hand belongs to the Decision-Hazard framework, and the non-anticipativity constraint writes
\begin{equation}
    \vPb \triangleleft_\delta \begin{pmatrix}
        \cons\\
        \pv
    \end{pmatrix}
\end{equation}

\begin{figure}[!t]
\centering
\includegraphics[width=\columnwidth]{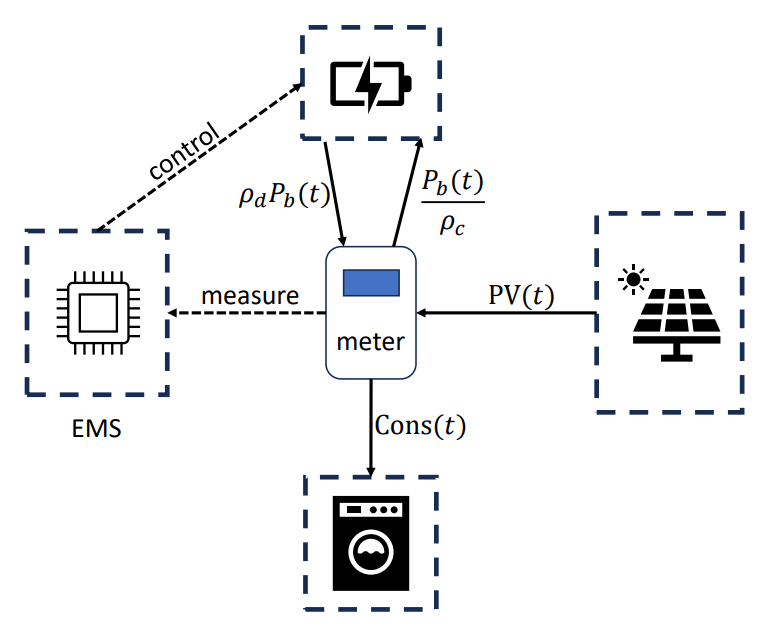}
\caption{schematic diagram of a domestic system with a stationary battery controlled by an EMS}
\label{fig:scheme_principle}
\end{figure}

\subsection{Rolling-horizon implementation}
In this section, we bring together, in a rolling-horizon framework, the VPPHA from \Cref{sec:reg_pha}, the scenario generation and scenario reduction methods from \Cref{sec:scenario_gen_and_red}. The control algorithm is described in \Cref{alg:alg_Planning}, where we denote $Q^{\rm meas},{\rm Cons}^{\rm meas}, {\rm PV}^{\rm meas}, P_b^{\rm meas}$ respectively the battery's state of energy, the electric consumption and photovoltaic production measured at the meter, and the battery charging power setpoint. These variables are all deterministic in the sense that they correspond to a particular realization of a stochastic process.
\begin{algorithm}[H]
\caption{vppha$(t_0,t_f,\delta,N_s, N_{\rm red}, H, \alpha)$}\label{alg:alg_Planning}
    \begin{algorithmic}
        \STATE $t\gets t_0$
        \WHILE {$t\leq t_f$}
            \STATE Measure $Q^{\rm meas}(t)$
            \STATE $r \gets {\rm modulus}(t-t_0, H)$
            \IF{$ r = 0$}
                \STATE $\cons_{t:t+24} \gets $gen$\_$scen$({\rm Cons}^{\rm meas}(t-\delta), N_s)$
                \vspace{5pt}
                \STATE $\overline{\cons}_{t:t+24} \gets $red$\_$scen$(\cons_{t:t+24}, N_{\rm red})$
                \vspace{5pt}
                \STATE $\pv_{t:t+24} \gets $gen$\_$scen$({\rm PV}^{\rm meas}(t-\delta), N_s)$
                \vspace{5pt}
                \STATE $\overline{\pv}_{t:t+24} \gets $red$\_$scen$(\pv_{t:t+24}, N_{\rm red})$
                \vspace{5pt}
                \STATE ${\vPb}_{t:t+24} \gets $VPPHA$(\alpha, \overline{\cons}_{t:t+24}, \overline{\pv}_{t:t+24}, Q^{\rm meas}(t))$
            \ENDIF
            \vspace{5pt}
            \STATE Compute $P_b^{\rm meas}(t)$ from $\vPb_{t-r:t-r+24}$, ${\rm Cons}^{\rm meas}(t-\delta)$, and ${\rm PV}^{\rm meas}(t - \delta)$
            \vspace{5pt}
            \STATE Measure ${\rm Cons}^{\rm meas}(t)$ and ${\rm PV}^{\rm meas}(t)$
            \vspace{5pt}
            \STATE $t \gets t + \delta$
    \ENDWHILE
    
    \STATE $P_m(t) := {\rm Cons}^{\rm meas}(t) - {\rm PV}^{\rm meas}(t) +\frac{1}{\rho_c} \max\{P_b^{\rm meas}(t), 0\}+ \rho_d \min\{P_b^{\rm meas}(t),0\} $
    \vspace{5pt}
    \STATE Bill $= \int_{t_0}^{t_f} {\rm pr}_{\rm b}(t) \max\{P_m(t),0\} +  {\rm pr}_{\rm s}(t) \min\{P_m(t),0\} \xd t$
    \vspace{5pt}
    \RETURN Bill
    \end{algorithmic}
\end{algorithm}

\subsection{Hyper parameter selection}
\Cref{alg:alg_Planning} requires to set 4 hyper-parameters, namely $\alpha, N_s, N_{\rm red},H$. The number of generated scenarios per random variable $N_s$ is set to $100$, and we set the rolling horizon to $H=24$ hours. We set $N_{\rm red}=15$, which yields a scenario tree with 225 branches. This number of scenarios is small enough to be numerically fast to solve and large enough to ensure the representativeness of the scenario tree. The last hyper-parameter $\alpha$ is determined by running \Cref{alg:alg_Planning} over 59 days, from 2024-05-04 to 2024-07-02, for different values of $\alpha$, and where ${\rm Cons}^{\rm meas}$ and ${\rm PV}^{\rm meas}$ are the ground truth measurements of electrical consumption and production. The buying price of electricity ${\rm pr}_{\rm b}$ is the day-ahead SPOT France, and the selling price ${\rm pr}_{\rm s}$ is set to 0. The performance of the proposed method is compared with a standard MPC strategy, which consists of setting $N_s=N_{\rm red}=1$, $\alpha=0$, and $H=0.5$ hour, i.e., only one scenario is generated, and the optimal control problem is solved every 30 minutes. Therefore the performance ratio denoted $\eta$ is defined as follows
\begin{equation}
\label{eq:def_performance_ratio}
    \eta(\alpha) :=
    100\left(\frac{\rho - {\rm vppha}(t_0,t_f,1/6,100, 15, 24, \alpha)}{\rho - {\rm vppha}(t_0,t_f,1/6, 1, 1, 0.5, 0)}-1\right)
\end{equation}
where $\rho$ is the reference bill defined as 
\begin{multline}
\label{eq:def_battery_less_bill}
    \rho := \int_{t_0}^{t_f}{\rm pr}_{\rm b}(t) \max\{{\rm Cons}^{\rm meas}(t)-{\rm PV}^{\rm meas}(t),0\} \\
    +  {\rm pr}_{\rm s}(t) \min\{{\rm Cons}^{\rm meas}(t)-{\rm PV}^{\rm meas}(t),0\} \xd t
\end{multline}
The results of these simulations are displayed on  \Cref{fig:hyp_param_24hrs}. One can see that the VPPHA with $\alpha>0$ always improves the performance ratio with respect to the standard PHA ($\alpha = 0$), and $\alpha=5$ seems to be the optimal value for the problem at hand.
\begin{figure}[!t]
\centering
\includegraphics[width=\columnwidth]{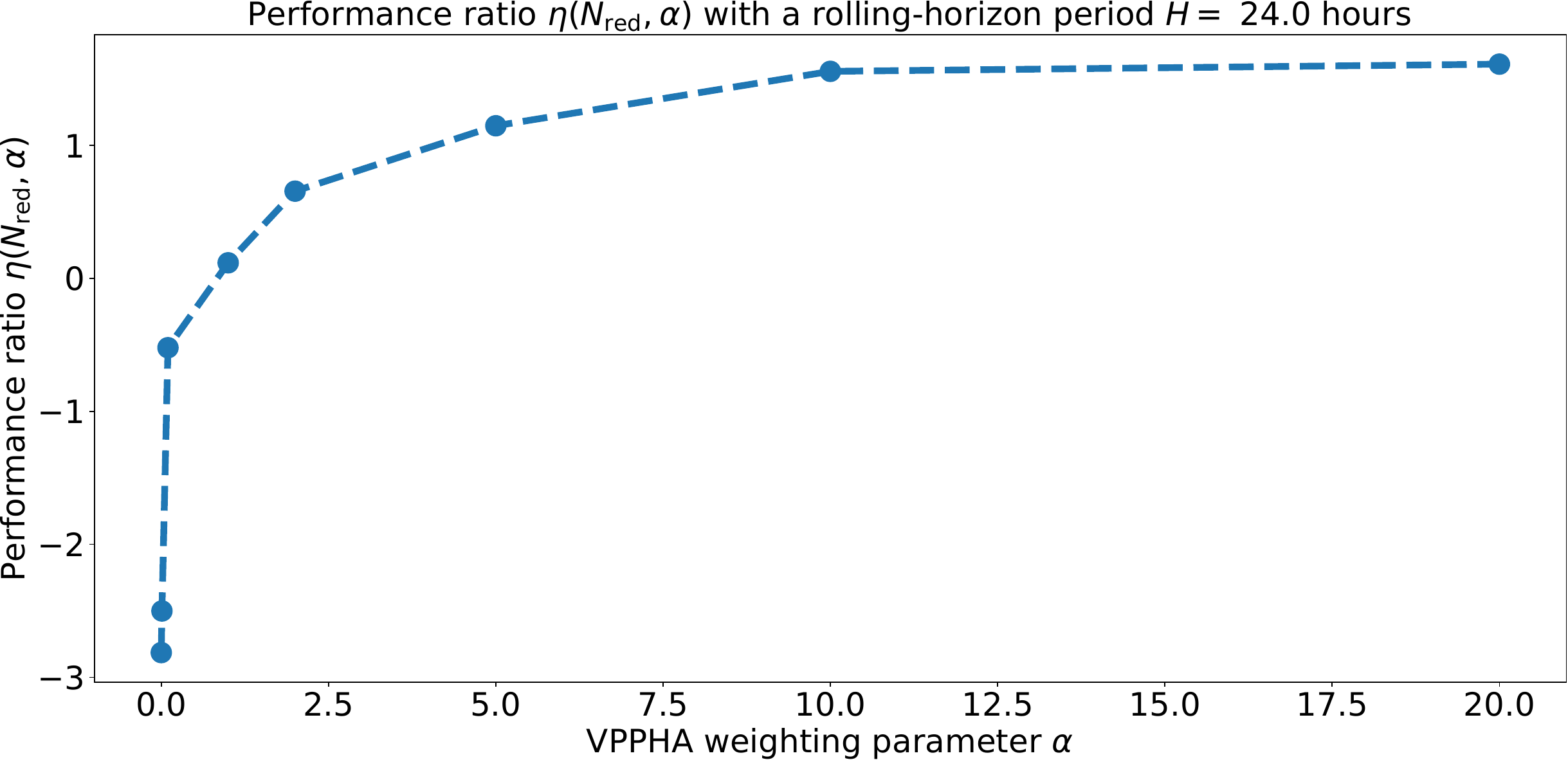}
\caption{Influence of the weighting parameter $\alpha$ on the performance ratio $\eta(\alpha)$ with an actualization period $H = 24$ hours and a scenario tree of 225-scenarios.}
\label{fig:hyp_param_24hrs}
\end{figure}

\subsection{Two years simulation}
Finally, we test and compare the performances of the VPPHA with a classical MPC strategy and the standard PHA over
two years ranging from 2022-01-22 to 2024-01-22. The parameterization of these different control strategies is displayed in \Cref{tab:mpc_pha_rpha}. In \Cref{fig:perf_ratio_two_years}, we compare the evolution of the performance ratio defined in \Cref{eq:def_performance_ratio} for the standard PHA and the VPPHA. This figure illustrates the standard PHA's lack of robustness. Indeed, the associated performance ratio converges to a negative value, i.e., it is less efficient than a classical MPC control strategy. On the contrary, the proposed VPPHA is more performant than the MPC strategy. Interestingly, one can notice an increase (resp. decrease) in efficiency for the VPPHA (resp. standard PHA) during the summer of 2022. During this period, SPOT electricity prices in France were unusually high due to limited availability of French nuclear power plants and high gas prices following the Russian invasion of Ukraine. Thus, an efficient control strategy must be risk-averse to avoid unnecessary, highly priced electricity consumption. From this point of view, the proposed VPPHA strategy is indeed more risk-averse than the standard PHA and also improves the performance of the EMS compared to the MPC strategy. Indeed, in \Cref{fig:bill_red_two_years}, we compare the electricity bill reduction provided by each control strategy compared to the battery-less electricity bill $\rho$ defined in \cref{eq:def_battery_less_bill}. At the end of the simulation, the MPC strategy allows for an electricity bill reduction of 7.30\%, the standard PHA allows for a bill reduction of 7.13\%, and the VPPHA allows for a bill reduction of 7.95\%. Therefore, the VPPHA strategy allows for a 0.65\% additional bill reduction compared with the standard MPC strategy while only requiring the resolution of a complex optimal control problem every 24 hours. In the meantime, the standard PHA performs less efficiently than the MPC.

\begin{table}[h]
    \begin{tabular}{r|c |c|c|c|c}
            Control Strategy & $\delta$ (hrs)& $H$ (hrs) & $N_s$ & $N_{\rm red}$ & $\alpha$ \\
        \hline
         MPC&  $1/6$ & 0.5   & 1   & 1 & 0\\
         Standard PHA&  $1/6$ & 24   & 100   & 15 & 0 \\
         VPPHA &  $1/6$ & 24   & 100   & 15 & 5
    \end{tabular}
    \caption{Control strategies hyper-parameters selection}
    \label{tab:mpc_pha_rpha}
\end{table}

\begin{figure}[!t]
\centering
\includegraphics[width=\columnwidth]{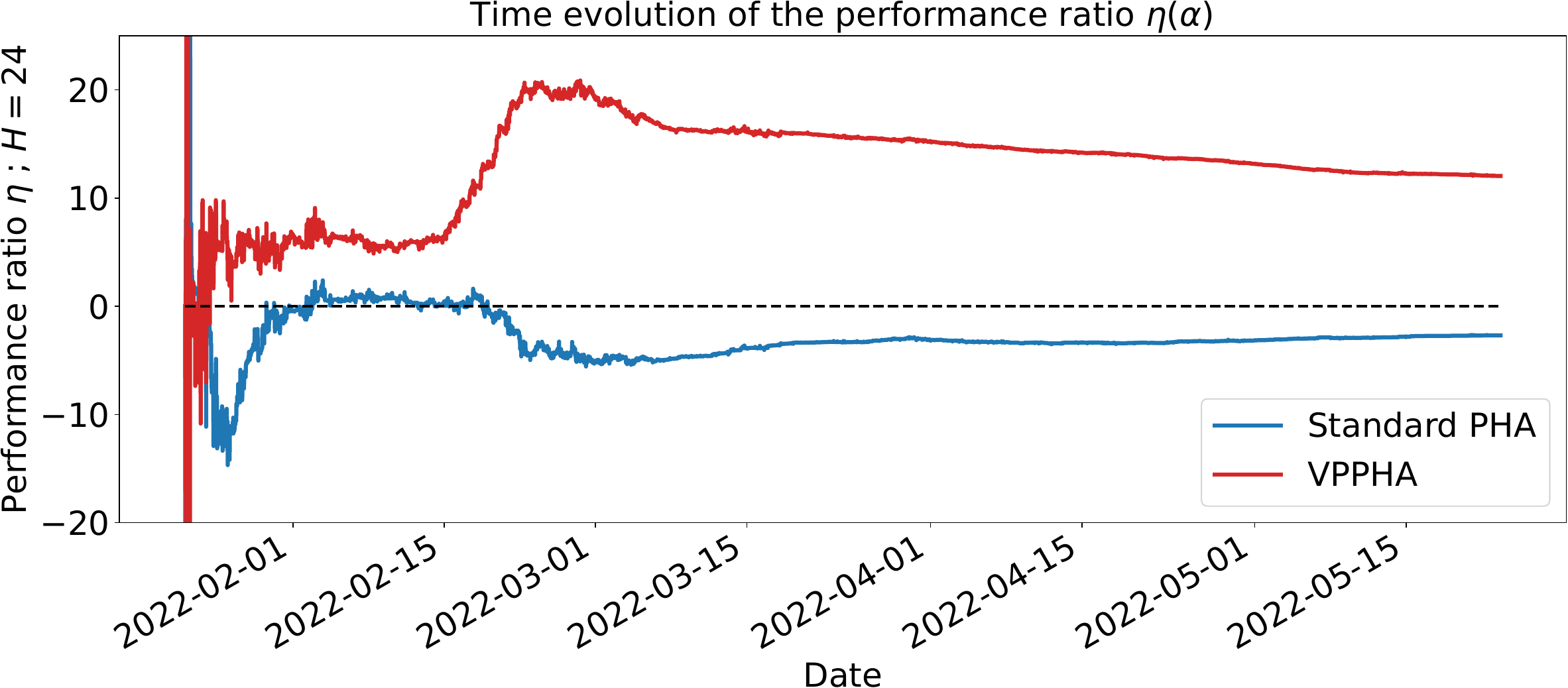}
\caption{time-evolution of the performance ratio $\eta(\alpha)$ from the 2022-01-22 to the 2024-01-22.}
\label{fig:perf_ratio_two_years}
\end{figure}

\begin{figure}[!t]
\centering
\includegraphics[width=\columnwidth]{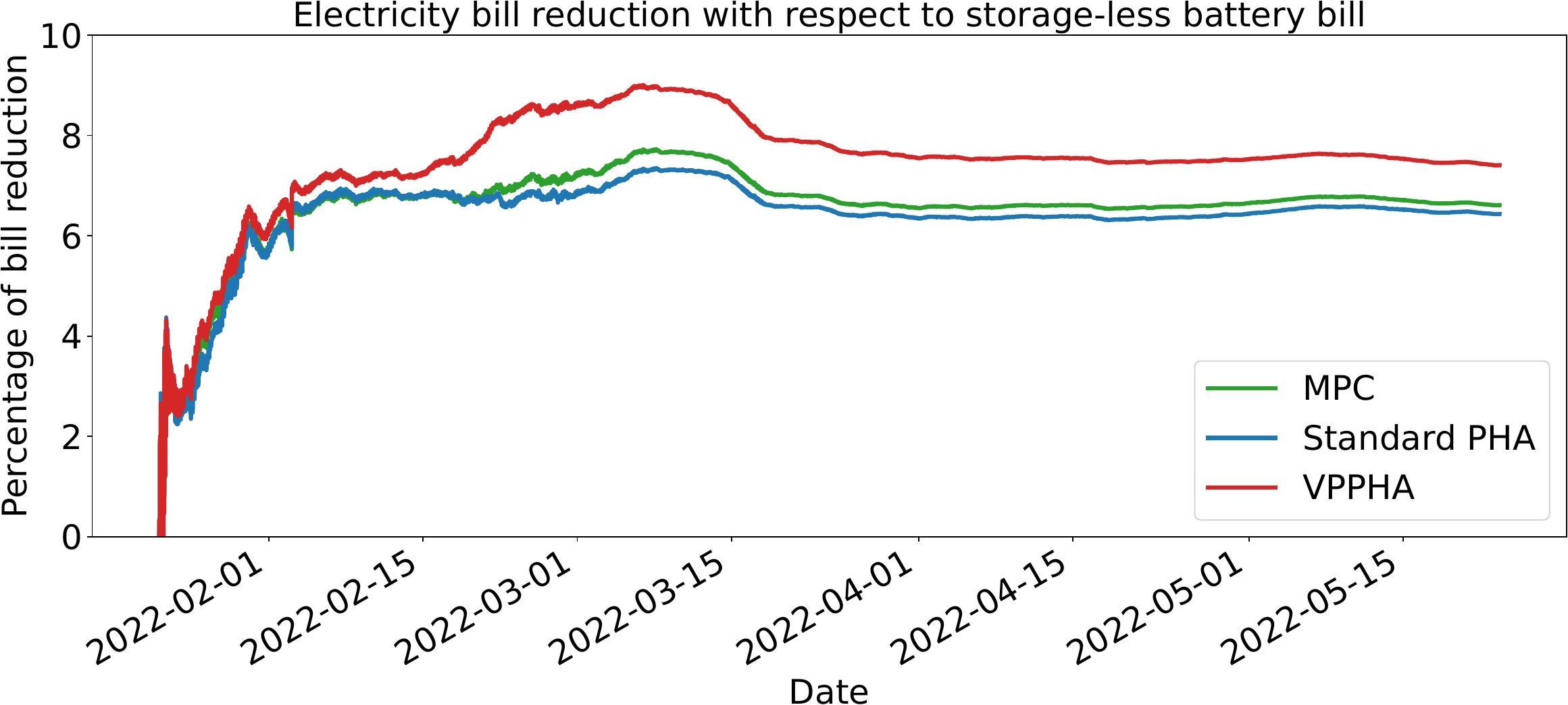}
\caption{Time-evolution of the percentage of electricity bill reduction from the 2022-01-22 to the 2024-01-22.}
\label{fig:bill_red_two_years}
\end{figure}

\section{Conclusion}
This article proposes a variance-regularized PHA, called VPPHA. This VPPHA has the same numerical complexity as the standard PHA but exhibits better out-of-sample performances. In addition, we have shown on actual data from an industrial site that the proposed framework, consisting of scenario generation, scenario reduction, and VPPHA, performs better than the standard PHA and a classical MPC strategy, making it a strong candidate for actual implementation in an EMS.

\section*{Declarations}
\subsection*{Conflict of Interest} 
Authors have no relevant financial or non-financial interests to disclose.
\subsection*{Data availability}
All experimental data are available at \url{10.5281/zenodo.17314269                     }.
\subsection*{Author contribution}
P.M. developed the optimization model and solving algorithm, and wrote Sections 1-4. A.S. developed the scenario generation method and wrote Section 5. All authors contributed equally to numerical experimentation and the writing of Section 6.

\bibliography{bibliography}

\end{document}